\documentclass{article}

\usepackage{cite}
\usepackage{xcolor}

\usepackage[centertags]{amsmath}
\usepackage{amssymb}
\usepackage{amsthm}
\usepackage{ifthen}

\usepackage{todonotes}

\newcommand{\vectorSequence}[1][n]{
\ifthenelse{\equal{#1}{0}}{\underline{\eta_{#1}}}{}
\ifthenelse{\equal{#1}{n}}{\underline{\eta_{n}}}{}
\ifthenelse{\equal{#1}{\underline{d}}}{\underline{\eta_{n,\textbf{d}}}}{}
}

\def\cB{{\mathcal B}}
\def\cC{{\mathcal C}}

\def\cL{{\mathcal L}}

\def\cR{{\mathcal R}}

\def\cT{{\mathcal T}}

\def\Tr{{\mathrm{Tr}}}

\newcommand{\F}{\mathbb{F}}

\newtheorem{theorem}{Theorem}
\newtheorem{lemma}{Lemma}

\newcommand{\vn}[1]{{#1}} 

\theoremstyle{remark}

\begin{document}
\title{On the distribution of the Rudin-Shapiro function for finite fields}
\author{
{\sc C\'ecile Dartyge$^1$, L\'aszl\'o M\'erai$^2$,  Arne Winterhof$^2$}\\
$^1$ Institut \'Elie Cartan, Universit\'e de Lorraine\\ BP 239, 54506 Vand\oe{}uvre Cedex, France\\
e-mail: cecile.dartyge@univ-lorraine.fr
 \\
$^2$ Johann Radon Institute for\\ 
Computational and Applied Mathematics\\
Austrian Academy of Sciences\\
Altenbergerstr.\ 69, 4040 Linz, Austria\\
e-mail: \{laszlo.merai,arne.winterhof\}@oeaw.ac.at}

\date{}

\maketitle

\begin{center}{\large In memory of Christian Mauduit}
\end{center}

\bigskip

\begin{abstract}
Let $q=p^r$ be the power of a prime $p$ and $(\beta_1,\ldots ,\beta_r)$ be an ordered basis of $\F_q$ over $\F_p$.
For 
$$\xi=\sum\limits_{j=1}^r x_j\beta_j\in \F_q \quad \mbox{with digits }x_j\in\F_p,$$
we define the Rudin-Shapiro function $R$ on $\F_q$ by 
$$R(\xi)=\sum\limits_{i=1}^{r-1} x_ix_{i+1}, \quad \xi\in \F_q.$$
For a non-constant polynomial $f(X)\in \F_q[X]$ and $c\in \F_p$ we study the number of solutions
$\xi\in \F_q$ of $R(f(\xi))=c$. 
If the degree $d$ of $f(X)$ is fixed, $r\ge 6$ and $p\rightarrow \infty$, the number of solutions is asymptotically $p^{r-1}$ for any $c$. The proof is based on the Hooley-Katz Theorem.
\end{abstract}

\bigskip



\section{Introduction}

In recent years, many spectacular results have been obtained on important problems combining some arithmetic properties of the integers and some conditions on their digits in a given basis, see for example \cite{Bou13,Bou15,DMR19,MR09,MR10,M16,M18,SW19,Sw20}.
In particular, Drmota, Mauduit and Rivat \cite{DMR19} and M\"ullner \cite{M18} showed that Thue-Morse sequence and Rudin-Shapiro sequence along squares are both normal, that is, each binary pattern of the same length appears asymptotically with the same frequency.

A natural question is to study analog problems in finite fields, see for example \cite{DMS15,DS13,DES16,G17,M19,P19,Sw18b,Sw18a,Sw18c}.
Many of these problems can be solved for finite fields although their analogs for integers are actually out of reach.

In particular, it is conjectured but not proved yet that the subsequences of the Thue-Morse sequence and Rudin-Shapiro sequence along any polynomial of degree $d\ge 3$ 
are normal, see \cite[Conjecture 1]{DMR19}.
Even the weaker problem of determining the frequency of $0$ and $1$ in the subsequence of the Thue-Morse sequence and Rudin-Shapiro sequence along any polynomial of degree $d\ge 3$ seems to be out of reach, see 
\cite[above Conjecture 1]{DMR19}.
However, the analog of the latter weaker problem for the Thue-Morse sequence in the finite field setting was settled by the first author and S\'ark\"ozy\vn{~\cite{DS13}}.

This paper deals with the following analog of the frequency problem for the Rudin-Shapiro sequence along polynomials. 

Let $q=p^r$ be the power of a prime $p$ 
and $\cB=( \beta_1,\ldots ,\beta_r)$ be an ordered basis of the finite field $\F_q$ over $\F_p$.
Then any $\xi\in\F_q$ has a unique representation  
$$
\xi=\sum_{j=1}^r x_j\beta_j \quad\mbox{with } x_j\in \F_p,\quad j=1,\ldots,r.
$$
The coefficients $x_1,\ldots , x_r$ are called the {\em digits} with respect to the basis $\cB$.

In order to consider the finite field analogue of the Rudin-Shapiro sequence along polynomial values, we define the {\em Rudin-Shapiro function} $R(\xi)$ for the finite field $\F_q$ with respect to the basis $\cB$ by
$$
R(\xi)=\sum_{i=1}^{r-1}x_ix_{i+1},\quad \xi=x_1\beta_1+\cdots +x_r\beta_r\in \F_q, \quad r\ge 2.
$$
For $f(X)\in\F_q[X]$ and $c\in\F_p$
we put 
$$
\cR (c,f)=\{\xi\in\F_q : R(f(\xi))=c\}.
$$
Our goal is to prove that the size of $\cR(c,f)$ is asymptotically the same for all~$c$.

Our main result is the following theorem.

\begin{theorem}
\label{thm:RS}
Let $f(X)\in\F_q[X]$ be of degree $d\ge 1$. 
For $c\in\F_p$ we have
$$
\left||\cR (c,f)|-p^{r-1}\right|\le C_{d,r} p^{(3r+1)/4 -h_{r,c} },
$$
where $h_{r,c}$ is defined by
\begin{equation*}
h_{r,c}= 
\left\{\begin{array}{ll}
        3/4, & r \mbox{ even and }c\ne 0, \\
        1/2, & r \mbox{ odd and }c\ne 0,\\
        1/4, & r \mbox{ even and }c=0,\\
        0, & r \mbox{ odd and } c=0,
        \end{array}\right.       
\end{equation*}
and $C_{d,r}$ is a constant depending only on $d$ and $r$. 
\end{theorem}

In particular, we have for fixed $d$,
$$\lim_{p\rightarrow\infty}\frac{|\cR(c,f)|}{p^{r-1}}=1
\quad \mbox{ for }c\ne 0 \mbox{ and }r\ge 4 \mbox{ or }c=0 \mbox{ and }r\ge 6.$$

For $d=1$, or more generally, for any permutation polynomial $f(X)$ of $\F_q$,
it is easy to see that
\begin{equation*}
|\cR(c,f)|=\left\{\begin{array}{ll}p^{r-1}-p^{\lfloor (r-1)/2\rfloor}, & c\ne 0,\\
p^{r-1}+p^{\lfloor(r+1)/2 \rfloor}-p^{\lfloor(r-1)/2 \rfloor}, & c=0,\end{array}\right. \quad r\ge 2.
\end{equation*}
For the convenience of the reader we will provide a very short proof in Section~\ref{d1}.
Hence, it remains to prove Theorem~\ref{thm:RS} for $d\ge 2$.

A commonly used idea, for example in \cite{DMS15}, to estimate the number of solutions of certain equations over finite fields  
is to apply the Weil bound. In some special situations the Deligne bound \cite[Th\'eor\`eme 8.4]{De74} provides stronger results. 
The Weil bound has the only 
condition~$d\ge 1$ but is too weak for our purpose.
The Deligne bound needs some more intricate
technical conditions which are not satisfied in our situation, see Section~\ref{sec:remarks}.
Our main tool is 
a generalization of Deligne's Theorem for projective surfaces~\cite{De74}, the Hooley-Katz Theorem \cite{H91}, see Lemma~\ref{HK} in Section~\ref{hoka} below.
The crucial steps in the proof are:
\begin{enumerate}
    \item Identify $R(f(X))$ with a multivariate polynomial of the form 
    $$Q(Y_0,\ldots,Y_{r-1})=\sum_{j,k=0}^{r-1}a_{j,k}f_j(Y_j)f_k(Y_k),$$
    which is done in Section~\ref{identity}. Note that this polynomial has coefficients in~$\F_q$.
    \item Estimate the dimensions of the singular loci, defined in Section~\ref{hoka} below, of $Q-c$ and its
homogeneous part of largest degree, see Lemma~\ref{lemma:main} below.
\item We complete the proof in Section~\ref{identity}. After a linear variable substitution, $Q$ is transformed 
to a polynomial $F$ of the same degree as $Q$ but with coefficients in $\F_p$. In particular, the dimensions of the singular loci are invariant under this linear transformation. Then we
     apply the Hooley-Katz Theorem to $F-c$.
\end{enumerate}

\section{The case of permutation polynomials}
\label{d1}

For a permutation polynomial $f(X)$ of $\F_q$, $|\cR(c,f)|$ is the number $N_r(c)$ of solutions $(x_1,\ldots,x_r)\in \F_p^r$ of the equation 
$$x_1x_2+\ldots+x_{r-1}x_r=c.$$
We have
\begin{equation*}
N_r(c)=\left\{\begin{array}{ll}p^{r-1}-p^{\lfloor (r-1)/2\rfloor}, & c\ne 0,\\
p^{r-1}+p^{\lfloor(r+1)/2 \rfloor}-p^{\lfloor(r-1)/2 \rfloor}, & c=0,\end{array}\right. \quad r\ge 2,
\end{equation*}
which can be easily verified using the recursion
$$N_r(c)=pN_{r-2}(c)+(p-1)p^{r-2},\quad r\ge 4.$$
This recursion is obtained by distinguishing the cases $x_{r-1}=0$ and $x_{r-1}\ne 0$.

\section{The Hooley-Katz Theorem}\label{hoka}

We denote by $\overline{\F_p}$ the algebraic closure of $\F_p$.

The {\it (affine) singular locus} $\cL(F)$ of a polynomial $F$ over $\F_p$ in $r$ variables
is the set of common zeros in $\overline{\F_p}^r$ of the polynomials 
$$F,\frac{\partial F}{\partial X_1},\ldots,\frac{\partial F}{\partial X_r}.$$

Our main tool is the following result, see \cite[Theorem~7.1.14]{HFF}, which is the affine version of the Hooley-Katz Theorem~\cite{H91}.

\begin{lemma}[Hooley-Katz]
\label{HK}
Let $F$ be a polynomial over $\F_p$ in $r$ variables of degree~$D\ge 1$ such that the dimensions of the singular loci of $F$ and its homogeneous part $F_D$ of degree $D$ satisfy
$$
\max\{\dim(\cL(F)),\dim(\cL(F_D))-1\}\le s.
$$ 
Then  the number $N$ of zeros of~$F$ in $\F_p^r$ satisfies
$$\left|N-p^{r-1}\right|\le C_{D,r}p^{(r+s)/2},$$
where $C_{D,r}$ is a constant depending only on $D$ and $r$.
\end{lemma}
We remark, that in the statement $\dim(\cL(F_D))$ denotes the  dimension of the \emph{affine} singular locus of the homogeneous polynomial $F_D$ while in \cite[Theorem~7.1.14]{HFF} the dimension of the \emph{projective} singular locus is considered. The difference of these dimensions is $1$.

\section{Proof of Theorem~\ref{thm:RS}}\label{identity}

First, we express the Rudin-Shapiro function $R(\xi)$ of $\F_q$ in terms of the trace and the dual basis.

Let $\varphi$ be the \emph{Frobenius automorphism} defined by 
$$\varphi(\xi)=\xi^p \quad \mbox{for }\xi\in\F_q.$$ 
We extend $\varphi$ to the polynomial ring $\F_q[X_1,\dots, X_{r}]$ by 
$$\varphi(X_i)=X_i,\quad i=1,\ldots,r.$$
Let
$$
\Tr( \xi)=\xi+\varphi(\xi)+\cdots +\varphi^{{r-1}}(\xi)\in\F_p
$$ 
denote the (absolute) \emph{trace} of $\xi\in \F_q$. 
Let $(\delta_1,\ldots,\delta_r)$ denote the (existent and unique) \emph{dual basis} of the basis~$\cB =( \beta_1,\ldots ,\beta_r)$ of $\F_q$, see for example \cite{LN97}, 
that is,  
 \begin{equation}\label{dual}
\Tr (\delta_i\beta_j)=
\begin{cases}
1 & {\rm if}\ i=j,\\
0 & {\rm if}\ i\not =j, 
\end{cases}
\quad 1\le i,j\le r.
\end{equation}
Then we have 
$$\Tr (\delta_i\xi)=x_i\quad\mbox{for any}\quad \xi=\sum_{j=1}^r x_j \beta_j\in\F_q\quad \mbox{with }x_j\in \F_p.
$$
For $f(X)\in \F_q[X]$ we obtain that
\begin{equation*}
\begin{split}
R(f(\xi))&=\sum_{i=1}^{r-1}\Tr(\delta_if(\xi))\Tr(\delta_{i+1}f(\xi))\\
&=\sum_{i=1}^{r-1}
\sum_{j,k=0}^{r-1} \varphi^j(\delta_i)\varphi^k(\delta_{i+1}) \varphi^j(f(\xi)) \varphi^k(f(\xi)).
\end{split}
\end{equation*}
Write
\begin{multline}\label{eq:F}
F(X_1,\dots, X_{r})\\
=\sum_{j,k=0}^{r-1} a_{j,k} f_j(\beta_1^{p^j}X_1+\dots+ \beta_{r}^{p^j}X_{r} )
 f_k(\beta_1^{p^k}X_1+\dots+ \beta_{r}^{p^k}X_{r} ),
\end{multline}
where
\begin{equation}\label{def:ajk}
a_{j,k}=\sum_{i=1}^{r-1}\varphi^j(\delta_i)\varphi^k(\delta_{i+1}),\quad j,k=0,\ldots,r-1,
\end{equation}
and
$f_j=\varphi^j(f)\in\F_q[X]$. 
Verify $\varphi(F)=F$, that is, $F\in\F_p[X_1,\dots, X_{r}]$ and 
$$
R(f(\xi))=F(x_1,\dots, x_{r})\quad \text{for} \quad \xi=\sum_{i=1}^{r}x_i\beta_i, \quad x_i\in\F_p.
$$

Theorem~\ref{thm:RS} follows from Lemma~\ref{HK}
and the following lemma which we prove in the next section.

\begin{lemma}\label{lemma:main}
Let $f(X)\in \F_q[X]$ be of degree $d$ with $2\le d <p$
and $F\in\F_p[X_1,\dots, X_{r}]$ be defined by \eqref{eq:F}. Then $F$ has degree $2d$. Moreover, for any~$c\in \F_p$ we have
$$
\dim(\cL(F-c))\le \left\{\begin{array}{ll} r/2-1, & r \mbox{ even and }c\ne 0,\\
(r-1)/2, & r \mbox{ odd and }c\ne 0,\\
r/2, & r \mbox { even and }c=0,\\
(r+1)/2, &  r\mbox{ odd and }c=0.\end{array}\right.
$$
Furthermore, if $F_{2d}\in \F_p[X_1,\dots, X_{r}]$ is the homogeneous part of $F$ of degree~$2d$, then
$$
\dim(\cL(F_{2d}))\le 
\left\{
\begin{array}{ll} 
r/2, & r \mbox { even},\\
(r+1)/2, &  r\mbox{ odd}.
\end{array}\right.
$$
\end{lemma}

\section{Proof of Lemma~\ref{lemma:main}}
Consider the linear transformation on $\overline{\F_p}^r$ 
\begin{equation*}
y_i=\sum_{j=1}^r\beta_j^{p^i}x_j,\quad i=0,\ldots,r-1.
\end{equation*}
It is invertible with inverse 
\begin{equation}\label{transdelta}
x_k=\sum_{i=0}^{r-1}\delta_k^{p^i}y_i,\quad k=1,\ldots,r,
\end{equation}
by \eqref{dual}.

Then we denote by $Q$ the polynomial obtained from $F$, defined by \eqref{eq:F}, with the corresponding variable transformation,
\begin{equation*}
F(X_1,\dots, X_{r})=\sum_{j,k=0}^{r-1} a_{j,k} f_j(Y_j)f_k(Y_k)\\
=Q(Y_0,\dots, Y_{r-1}),
\end{equation*}
where 
\begin{equation}\label{eq:coord}
Y_{i}=\sum_{j=1}^{r}\beta_j^{p^i}X_j, \quad i=0,\dots, r-1.
\end{equation}
As the degree and the dimension, see \cite[Corollary~9.5.3]{CLO15}, of singular loci are invariant under the regular transformation~\eqref{eq:coord}, it is enough to show the results for the polynomial $Q$.

We may assume that $f(X)$ is monic since otherwise we multiply the basis $\cB$ element-wise with the leading coefficient of $f(X)$.
The degree $2d$ homogeneous part of $Q$ is
$$
Q_{2d}(Y_0,\ldots,Y_{r-1})=\sum_{j,k=0}^{r-1} a_{j,k} Y_j^dY_k^d.
$$

By the definition \eqref{def:ajk} of $a_{j,k}$ we have $$\sum_{j=0}^{r-1}a_{j,0}\beta_1^{p^j}=\sum_{i=1}^{r-1}\delta_{i+1}\Tr(\beta_1\delta_i)=\delta_2\ne 0.$$ Hence, $a_{j,0}\ne 0$ for some $j$. 
Since $Y_j^dY_k^d$, $0\leq j,k<r$, are linearly independent over $\F_q$, we get that $Q_{2d}$ is not the zero polynomial.
In particular we have
$$\deg(F)=\deg(Q)=\deg(Q_{2d})=2d.$$

We estimate the dimension of the singular locus $\mathcal{L}(Q-c)$. 
The bound for the dimension of $\mathcal{L}(Q_{2d})$ corresponds to the special case $f(X)=X^d$ and $c=0$,
that is, $Q=Q_{2d}$ in this case.

To estimate $\dim(\mathcal{L}(Q-c))$, consider the partial derivatives
$$
\frac{\partial (Q-c)}{\partial Y_\ell}(Y_0,\ldots,Y_{r-1})
=f_\ell'(Y_\ell)\sum_{k=0}^{r-1} (a_{k,\ell}+a_{\ell,k})f_k(Y_k), \quad  \ell=0,\dots,r-1.
$$
The condition $2\le d<p$ implies that $f'(X)=f_0'(X)$ is not constant and so~$f_\ell'(X)$ is not constant for $\ell=0,\ldots,r-1$.

Note that
$$\cL(Q-c)=\bigcup_{L\subseteq\{0,\ldots,r-1\}}(V_L\cap C_L ),$$ 
where $V_L$ is the (affine) variety in $\overline{\F_p}^r$ of solutions of the system of equations
\begin{equation}\label{eq:def_VL}
\begin{split}
Q(Y_0,\ldots,Y_{r-1})&=c,\\
\sum_{k=0}^{r-1} (a_{k,\ell}+a_{\ell,k})f_k(Y_k)&=0,\quad \ell \in L,
\end{split}
\end{equation}
and $C_L$ the variety of solutions of
\begin{align*} 
Q(Y_0,\ldots,Y_{r-1})&=c,\\
f_\ell'(Y_\ell)&=0,\quad \ell\in \{0,1,\ldots,r-1\}\setminus L.
\end{align*}
Hence, 
\begin{equation}\label{dimcup} \dim(\cL(Q-c))\le \max\{ \min\{\dim(V_L),\dim(C_L)\}: L\subseteq \{0,\ldots,r-1\}\},
\end{equation}
since 
$$\dim(U\cup V)=\max\{\dim(U),\dim(V)\}$$
and 
$$\dim(U\cap V)\le\min\{\dim(U),\dim(V)\},$$
see for example \cite[Propositions 9.4.8 and 9.4.1]{CLO15}.

It remains to estimate the dimensions of $V_L$ and $C_L$.

\begin{lemma}\label{mainlemma}
For $L\subseteq\{0,1,\ldots,r-1\}$ the (affine) variety $V_L$ 
is of dimension at most 
\vn{$$
\left\{\begin{array}{ll}
 r-|L|-1, & r \mbox{ even and }c\ne 0, \\
        r-|L|, & r \mbox{ even and }c= 0 \mbox{ or }
        r \mbox{ odd and }c\ne0,\\
        r-|L|+1, & r \mbox{ odd and } c=0.
\end{array}
\right.$$}
\end{lemma}
\begin{proof}
\vn{
For
any  $L\subset \{ 0,\dots ,r-1\}$ we consider the variety $W_L$ obtained by replacing  $Z_j=f(Y_j)$ for $j=0,\ldots,r-1$ in the defining equations \eqref{eq:def_VL} of $V_L$.
The variety $W_L$ 
is 
 the set of solutions $(\zeta_0,\ldots,\zeta_{r-1})\in\overline{\F_p}^r$ of the system
 \begin{align}\label{ZjZk}
\sum_{j,k=0}^{r-1}a_{j,k}Z_jZ_k&=c\\ \nonumber
\sum_{k=0}^{r-1} (a_{k,\ell}+a_{\ell,k})Z_k&=0,\quad \ell \in L.
\end{align}

First we show 
\begin{equation}\label{VLWL}
\dim(V_L)\le\dim(W_L).
\end{equation}
Put $s=\dim(W_L)$. Since otherwise \eqref{VLWL} is trivial we may assume $s<r$.
By \cite[Corollary 9.5.4]{CLO15}
for all $\{i_1,\ldots,i_{s+1}\}\subseteq\{0,\ldots ,r-1\}$, there exists a nonzero polynomial $P$ in $s+1$ variables with coefficients in $\overline{\F_p}$ such that 
$$
  P(\zeta_{i_1},\ldots , \zeta_{i_{s +1}})=0\quad \mbox{for all }(\zeta_0,\ldots ,\zeta_{r-1})\in W_L
$$ 
and thus  
$$
P(f_{i_1}(\eta_{i_1}),\ldots ,f_{i_{s +1}}(\eta_{i_{s+1}}))=0 \quad\mbox{for all }(\eta_0,\ldots ,\eta_{r-1})\in V_L.
$$
Since the polynomial $F(Y_{i_1},\ldots,Y_{i_{s+1}})=P(f_{i_1}(Y_{i_1}),\ldots,f_{i_{s+1}}(Y_{i_{s+1}}))$ is obviously not the zero polynomial
we deduce 
$\dim(V_L)\le s$
by \cite[Corollary~9.5.4]{CLO15}.
It remains to show
\begin{equation*}
\dim(W_L)\le 
\left\{\begin{array}{ll}
 r-|L|-1, & r \mbox{ even and }c\ne 0, \\
        r-|L|, & r \mbox{ even and }c= 0 \mbox{ or }r \mbox{ odd and }c\ne0,\\
        r-|L|+1, & r \mbox{ odd and } c=0.
\end{array}
\right.
\end{equation*}
Let $\widetilde W_L$ be the $\overline{\F_p}$-linear space of solutions of the system of linear equations 
$$\sum_{k=0}^{r-1} (a_{k,\ell}+a_{\ell,k})Z_k=0,\quad \ell \in L.$$

First we compute $\dim(\widetilde W_{\{0,1,\ldots,r-1\}})$, that is, we determine $(\zeta_0,\ldots,\zeta_{r-1})\in \overline{\F_p}^r$ satisfying 
$$
\sum_{k=0}^{r-1} (a_{\ell,k}+a_{k,\ell})\zeta_k=0,\quad \ell=0,\ldots,r-1,
$$}
and thus
$$
\sum_{k=0}^{r-1}\sum_{\ell=0}^{r-1}\beta_m^{p^\ell} (a_{\ell,k}+a_{k,\ell})\zeta_k=0,\quad m=1,\ldots,r.
$$
Since
\begin{align*}
\sum_{\ell=0}^{r-1}\beta_m^{p^\ell} (a_{\ell,k}+a_{k,\ell})
&= \sum_{i=1}^{r-1}\left(\delta_{i+1}^{p^k}\Tr(\beta_m\delta_i)+\delta_i^{p^k}\Tr(\beta_m\delta_{i+1})\right)\\
&=
\left\{\begin{array}{cc} \delta_2^{p^k}, & m=1,\\
\delta_{m-1}^{p^k}+\delta_{m+1}^{p^k}, & m=2,\ldots,r-1,\\
\delta_{r-1}^{p^k},& m=r,\end{array}\right.
\end{align*}
for $k=0,\ldots,r-1$,
we get 
\begin{equation}\label{eq:LER}
\begin{split}
\sum_{k=0}^{r-1}\delta_2^{p^k}\zeta_k &=0,\\
\sum_{k=0}^{r-1}\left(\delta_{m-1}^{p^k}+\delta_{m+1}^{p^k}\right)\zeta_k &=0,\quad m=2,\ldots,r-1,\\
\sum_{k=0}^{r-1}\delta_{r-1}^{p^k}\zeta_k&=0.
\end{split}
\end{equation}
For even $r$ this implies
$$
\sum_{k=0}^{r-1}\delta_m^{p^k}\zeta_k=0,\quad m=1,\ldots,r,
$$
and since the transformation $\eqref{transdelta}$ is regular
we get $\zeta_k=0$ for all $k$ and thus \vn{$\dim(\widetilde{W}_{\{0,\ldots,r-1\}})=0$}.

For odd $r$, \eqref{eq:LER} implies 
\begin{equation*}
\sum_{k=0}^{r-1}\delta_m^{p^k}\zeta_k=\left\{\begin{array}{ll} 0, & m\mbox{ even}, \\
(-1)^{(m-1)/2}\lambda, & m \mbox{ odd},\end{array}\right. \quad m=1,\ldots,r,
\end{equation*}
where $\sum_{k=0}^{r-1}\delta_1^{p^k}\zeta_k=\lambda$ for some $\lambda\in \overline{\F_p}$. We get 
\vn{$\dim(\widetilde{W}_{\{0,\ldots,r-1\}})=1$.

Now for any proper subset $L$ of $\{0,\ldots,r-1\}$ the variety is defined by deleting $j=r-|L|$ equations from the
definition of $\widetilde{W}_{\{0,\ldots,r-1\}}$. That is, its dimension is
increased by at most $r-|L|$.
The vector space $\widetilde{W}_L$
is  of dimension $t$ with $t\le r-|L|$ for even $r$ and $t\le \min \{ r, r-|L|+1\}$ for odd $r$. }

\vn{
Let $(u_1,\ldots ,u_{t})$ be basis of $\widetilde W_L$ so that each $z=(\zeta_0,\ldots ,\zeta_{r-1})\in \widetilde W_L$
is of the form $z=\sum_{i=1}^{t} \lambda_iu_i$ with $\lambda_1,\ldots ,\lambda_{t}\in \overline{\F _p}$.
If we write each $u_i=(\mu_{0,i},\ldots ,\mu_{r-1 ,i})$
for $i=1,\ldots ,t$,
then the coordinates of $z$ satisfy 
$$
\zeta_k=\lambda _1\mu_{k,1}+\cdots +\lambda _{t}\mu_{k,t} \quad \text{for } 
0\le k\le r-1.
$$
After the linear variable substitution
$$Z_k=\sum_{i=1}^t \mu_{ki} L_i$$
the  first equation of \eqref{ZjZk}
becomes
\begin{equation*}
\sum_{j,k=0}^{r-1}a_{j,k} \left(\sum_{i=1}^{t}  \mu_{j,i}L_i\right)\left(\sum_{i=1}^{t} \mu_{k,i}L_i\right)
=c.
\end{equation*}
The left hand side is a quadratic form in $L_1,\ldots ,L_{t}$.
If this form is identically zero, then $W_L=\emptyset$ if $c\not =0$ and $W_L=\widetilde W_L$ if $c=0$.
If the form is not identically zero, then
the variables $L_1,\ldots,L_t$ are algebraically dependent and we get
$\dim(W_L)\le\dim(\widetilde{W}_L) -1$
by \cite[Corollary 9.5.4]{CLO15}.
}\end{proof}

It is easy to see that 
\begin{equation}\label{CL} \dim(C_L)\le |L|
\end{equation}
\vn{by removing the equation $Q(Y_0,\ldots ,Y_{r-1})=0$ and having the same argument as in the proof of Lemma \ref{mainlemma},
this time substituting $Z_j= f'_j (Y_j)$ for $j=0,\ldots ,r-1$ and applying \cite[Corollary 9.5.4]{CLO15}.}

Combining \eqref{dimcup}, Lemma~\ref{mainlemma} and \eqref{CL} we get
$$\dim(\cL(Q-c))\le \left\{\begin{array}{ll} r/2-1, & r \mbox{ even and }c\ne 0,\\
(r-1)/2, & r \mbox{ odd and }c\ne 0,\\
r/2, & r \mbox { even and }c=0,\\
(r+1)/2, &  r\mbox{ odd and }c=0.\end{array}\right.$$


\section{Final remarks} \label{sec:remarks}
\subsection*{Some cases with singular locus $\cL(Q_{2d})$ of positive dimension}

Unfortunately, we cannot apply the Deligne bound to obtain a better result if the
singular locus $\cL(Q_{2d})$ has positive dimension.

It is clear from the proof of Lemma~\ref{mainlemma} that for odd $r$ the singular locus of~$Q_{2d}$ is of dimension at least $1$.
For some special choices of the dual basis, $\mathcal{L}(Q_{2d})$ has also positive dimension 
for any $r\ge 4$.

Namely, if  
\begin{equation}\label{sumdelt}\sum_{i=1}^{r-1}\delta_i\delta_{i+1}=0,
\end{equation}
the coefficients $a_{j,j}$, $j=0,\ldots,r-1$, defined by \eqref{def:ajk} vanish. Then each $(\eta_0,\ldots,\eta_{r-1})\in \overline{\F_p}^r$ with only one non-zero coordinate is 
a singular point of $Q_{2d}$.

Now we construct such a dual basis.
Let $\alpha$ be a defining element of $\F_q$ over $\F_p$, that is, $\F_q=\F_p(\alpha)$. 
Then, for sufficiently large $p$ with respect to $r$, $(\delta_1,\ldots,\delta_r)$ defined by 
\begin{align*}
\delta_{2i+1}&=\alpha^{r-1-i},\quad i=0,1,\ldots,\lfloor r/2\rfloor -1,\\
\delta_{2i+2}&=\alpha^i,\quad i=0,1,\ldots,\lfloor (r-1)/2\rfloor -1,\\
\delta_r&=-
\left\{
\begin{array}{ll}(r/2-1)(\alpha^{r/2-1}+\alpha^{r/2-2}),& r \mbox{ even},\\
\frac{r-1}{2} \alpha^{(r+1)/2}+\frac{r-3}{2} \alpha^{(r-1)/2}, & r \mbox{ odd},
\end{array}
\right. \quad r\ge 4,
\end{align*}
is a basis of $\F_q$ over $\F_p$, since $\alpha^{\lfloor (r-1)/2\rfloor}$ appears only in $\delta_r$, satisfying
\eqref{sumdelt}.

\subsection*{The Thue-Morse function of $\F_q$ for monomials}
The {\em Thue-Morse function} $T$ for $\F_q$ with respect to the basis ${\cal B}$ is 
$$T(\xi)=\sum_{i=1}^r x_i,\quad \xi=x_1\beta_1+\ldots+x_r\beta_r\in \F_q,$$
where $x_1,\ldots,x_r\in \F_p$.
For $f(X)\in \F_q[X]$ of degree $d\ge 1$ and $c\in\F_p$
we put 
$$
\cT (c,f)=\{\xi\in\F_q : T(f(\xi))=c\}.
$$
The first author and S\'ark\"ozy \cite[Theorem~1.2]{DS13} proved 
$$\left||\cT(c,f)|-p^{r-1}\right|\le (d-1)p^{r/2},\quad \gcd(d,p)=1.$$
For monomials $f(X)=X^d$, $c\ne 0$, fixed $d\ge 2$ and fixed $r$, the Hooley-Katz Theorem provides the improvement
$$\left||\cT(c,X^d)|-p^{r-1}\right|\le C_{d,r}p^{(r-1)/2},\quad c\ne 0.$$

In particular, we get
$$\lim_{p\rightarrow \infty}\frac{|\cT(c,X^d)|}{p^{r-1}}=1,\quad c\ne 0,$$
also for $r=2$.

The crucial step is to show that the singular locus of 
$$Q(Y_0,\ldots,Y_{r-1})=\sum_{\ell=0}^{r-1} \delta^{p^\ell}Y_\ell^d-c$$
is of dimension $-1$,
where we used the same notation as before and 
$$\delta=\sum_{i=1}^r\delta_i\ne 0,$$
since $\delta_1,\ldots,\delta_r$ are linearly independent.
Now the partial derivatives are
$$\frac{\partial Q}{\partial Y_\ell}=\delta^{p^\ell}d Y_\ell^{d-1},\quad \ell=0,\ldots,r-1.$$
We may assume $d<p$. Then the only common zero of all partial derivatives is~$(0,\ldots,0)$. However, $(0,\ldots,0)$  is not a zero of $Q$ for $c\ne 0$.

\medskip
The Hooley-Katz Theorem can also be applied for general $f(X)\in \F_q[X]$ of degree $d\ge 2$ but would give an improvement of \cite[Theorem~1.2]{DS13} only for $c\in\F_p\setminus\cC$ where ${\cC}$ is a subset  of $\F_p$ with at most
$(d-1)^r$ elements, where~$d$ and~$r$ are fixed and $p$ is sufficiently large. The polynomial $Q$ for a general $f$ becomes 
\begin{equation*}
    Q(Y_0,\ldots ,Y_{r-1})=\sum_{\ell =0}^{r-1}
    \delta^{{p^\ell}}
    f_\ell (Y_\ell)-c,
    \end{equation*}
with $f_\ell =\varphi ^\ell (f)$ as in Section \ref{identity}.

A singular point $(\eta_0, \ldots ,\eta_{r-1})\in \overline{\F_p}^r$ satisfies 
\begin{equation}\label{fprime} f'_\ell (\eta_\ell )=0\quad\mbox{for } \ell =0,\ldots ,r-1.
\end{equation}
This singular point has to be also a zero of $Q$, that is,  
$$c= \sum_{i=0}^{r-1}\delta^{p^\ell}f_\ell
(\eta_\ell).$$
For all other $c\in \F_p$ there are no singular points.
Since \eqref{fprime} has at most $(d-1)^r$ solutions in $\overline{\F_p}^r$ we have $|\cC|\le (d-1)^r$.

\section*{Acknowledgments}

The second and third author are partially supported by the Austrian Science Fund FWF,
Projects P 31762 and P 30405, respectively. 

The authors wish to thank Igor Shparlinski for pointing to Deligne's bound for projective surfaces and related theorems.

We wish to thank the anonymous referees for very useful comments.

\bibliographystyle{plain}
\bibliography{newbib}

\end{document}